\documentclass[12pt,oneside]{article}
\usepackage{amsmath,amssymb,amsfonts,amsthm}
\usepackage{centernot}
\usepackage{stmaryrd}
\usepackage{lscape}
\newcommand{\T}{{\cal T}}

\newcommand{\set}[1]{\left\{#1\right\}}

\newcommand {\cp}{\mathfrak{X}(\pi (M))}

\def\Section#1{\vspace{30truept}\addtocounter{section}{1}\setcounter{thm}{0}
\setcounter{equation}{0}{\noindent\Large\bf
    \arabic{section}.~~#1}\par \vspace{12pt}}

\newtheorem{thm}{Theorem}[section]
\newtheorem{cor}[thm]{Corollary}

\newtheorem{defn}[thm]{Definition}

\newtheorem{rem}[thm]{Remark}

\numberwithin{equation}{section}
\begin{document}
\title{{{\bf On Hyper-Generalized Recurrent Finsler Spaces }}}%\footnote{ArXiv Number: ------ [math.DG]}}
\author{{\bf  A. Soleiman}}
\date{}
%\thanks{\it Department of Mathematics, etc}
%\pagestyle{fancy}

             % End of preamble and beginning of text.
\maketitle                     % Produces the title.
\vspace{-1.15cm}

\begin{center}
{Department of Mathematics, Faculty of Science,\\ Benha University, Benha,  Egypt}
\end{center}
\vspace{-0.8cm}
\begin{center}
amrsoleiman@yahoo.com, amr.hassan@fsci.bu.edu.eg
\end{center}
\smallskip

\vspace{1cm} \maketitle
\smallskip
\noindent{\bf Abstract.}
The aim of the present paper is to investigate new types of recurrence in  Finsler geometry, namely, hyper-generalized recurrence and generalized conharmonic recurrence. The properties of such recurrences and  their relations to other  Finsler recurrences are studied.


\bigskip

%%%%%%%%%%%%%%%%%%%%%%%%%%%%%%%%%%%%%%%%%%%%%%%%%%% Introduction %%%%%%%%%%%%%%%%%%%%%%%%%%%%%%%%%%%%%%%%%%%%%%%%%%%%%%%%

\begin{center}
\large{\bf Introduction}.
\end{center}
Many types of recurrence in Riemannian geometry  have been studied by many authors \cite{R3,  R4, R2, R1, R3a, R5}.
 On the other hand, some types of recurrence in Finsler geometry have been also studied \cite{F3, F2, amr3, F1, Types}.
\par
In  recent papers \cite{amr3, Types}, we have introduced and investigated \textmd{intrinsically} four classes of recurrence:
recurrence under projective change, simple recurrence, Ricci recurrence and concircular recurrence. Each of these classes consists of several types of recurrence.  We also investigated the interrelationships between the different types of recurrence.
\par
The present paper is a continuation of ~\cite{amr3, Types}, where we introduce and investigate other new types of  Finsler recurrences, namely, hyper-generalized recurrence and generalized conharmonic recurrence. The properties of such recurrences and  their relations to other  Finsler recurrences are obtained.  The results of this paper globalize and generalize some results of \cite{R1}.

\newpage
%%%%%%%%%%%%%%%%%%%%%%%%%%%%%%%%%%%%%% SECTION 1. Notation and Preliminaries %%%%%%%%%%%%%%%%%%%%%%%%%%%%%%%%%%%%%%

\Section{Notation and Preliminaries}

In this section, we give a brief account of the basic concepts
 of the pullback approach to intrinsic Finsler geometry necessary for this work. For more
 details, we refer to \cite{r58, A. Tamim,  r86, Cart. Ber., r94, r96}. We
 shall use the notations of \cite{r86}.

 In what follows, we denote by $\pi: \T M\longrightarrow M$ the subbundle of nonzero vectors
tangent to $M$, $\mathfrak{F}(TM)$ the algebra of $C^\infty$ functions on $TM$, $\cp$ the $\mathfrak{F}(TM)$-module of differentiable sections of the pullback bundle $\pi^{-1}(T M)$.
The elements of $\mathfrak{X}(\pi (M))$ will be called $\pi$-vector
fields and will be denoted by barred letters $\overline{X} $. The
tensor fields on $\pi^{-1}(TM)$ will be called $\pi$-tensor fields.
The fundamental $\pi$-vector field is the $\pi$-vector field
$\overline{\eta}$ defined by $\overline{\eta}(u)=(u,u)$ for all
$u\in \T M$.
\par
We have the following short exact sequence of vector bundles
$$0\longrightarrow
 \pi^{-1}(TM)\stackrel{\gamma}\longrightarrow T(\T M)\stackrel{\rho}\longrightarrow
\pi^{-1}(TM)\longrightarrow 0 ,\vspace{-0.1cm}$$ with the well known
definitions of  the bundle morphisms $\rho$ and $\gamma$. The vector
space $V_u (\T M)= \{ X \in T_u (\T M) : d\pi(X)=0 \}$  is the vertical space to $M$ at $u$.
\par
Let $D$ be  a linear connection on the pullback bundle $\pi^{-1}(TM)$.
 We associate with $D$ the map \vspace{-0.1cm} $K:T \T M\longrightarrow
\pi^{-1}(TM):X\longmapsto D_X \overline{\eta} ,$ called the
connection map of $D$.  The vector space $H_u (\T M)= \{ X \in T_u
(\T M) : K(X)=0 \}$ is called the horizontal space to $M$ at $u$ .
   The connection $D$ is said to be regular if
$$ T_u (\T M)=V_u (\T M)\oplus H_u (\T M) \,\,\,  \forall \, u\in \T M.$$

If $M$ is endowed with a regular connection, then the vector bundle
   maps $
 \gamma,\, \rho |_{H(\T M)}$ and $K |_{V(\T M)}$
 are vector bundle isomorphisms. The map
 $\beta:=(\rho |_{H(\T M)})^{-1}$
 will be called the horizontal map of the connection
$D$.
\par
 The horizontal ((h)h-) and
mixed ((h)hv-) torsion tensors of $D$, denoted by $Q $ and $ T $
respectively, are defined by \vspace{-0.2cm}
$$Q (\overline{X},\overline{Y})=\textbf{T}(\beta \overline{X}\beta \overline{Y}),
\, \,\,\, T(\overline{X},\overline{Y})=\textbf{T}(\gamma
\overline{X},\beta \overline{Y}) \quad \forall \,
\overline{X},\overline{Y}\in\mathfrak{X} (\pi (M)),\vspace{-0.2cm}$$
where $\textbf{T}$ is the (classical) torsion tensor field
associated with $D$.
\par
The horizontal (h-), mixed (hv-) and vertical (v-) curvature tensors
of $D$, denoted by $R$, $P$ and $S$
respectively, are defined by
$$R(\overline{X},\overline{Y})\overline{Z}=\textbf{K}(\beta
\overline{X}\beta \overline{Y})\overline{Z},\quad
 {P}(\overline{X},\overline{Y})\overline{Z}=\textbf{K}(\beta
\overline{X},\gamma \overline{Y})\overline{Z},\quad
 {S}(\overline{X},\overline{Y})\overline{Z}=\textbf{K}(\gamma
\overline{X},\gamma \overline{Y})\overline{Z}, $$
 where $\textbf{K}$
is the (classical) curvature tensor field associated with $D$.
\par
The contracted curvature tensors of $D$, denoted by $\widehat{{R}}$, $\widehat{ {P}}$ and $\widehat{ {S}}$ (known
also as the (v)h-, (v)hv- and (v)v-torsion tensors respectively), are defined by
$$\widehat{ {R}}(\overline{X},\overline{Y})={ {R}}(\overline{X},\overline{Y})\overline{\eta},\quad
\widehat{ {P}}(\overline{X},\overline{Y})={
{P}}(\overline{X},\overline{Y})\overline{\eta},\quad \widehat{
{S}}(\overline{X},\overline{Y})={
{S}}(\overline{X},\overline{Y})\overline{\eta}.$$
%\par
%The following result is of extreme importance. \vspace{-0.1cm}
\begin{thm} {\em\cite{r94}} \label{th.1} Let $(M,L)$ be a Finsler
manifold and  $g$ the Finsler metric defined by $L$. There exists a
unique regular connection $\nabla$ on $\pi^{-1}(TM)$ such
that\vspace{-0.2cm}
\begin{description}
  \item[(a)]  $\nabla$ is  metric\,{\em:} $\nabla g=0$,

  \item[(b)] The (h)h-torsion of $\nabla$ vanishes\,{\em:} $Q=0
  $,
  \item[(c)] The (h)hv-torsion $T$ of $\nabla$\, satisfies\,\textmd{:}
   $g(T(\overline{X},\overline{Y}), \overline{Z})=g(T(\overline{X},\overline{Z}),\overline{Y})$.
\end{description}
\par
 Such a connection is called the Cartan
connection associated with  the Finsler manifold $(M,L)$.
\end{thm}
The only linear connection we deal with in this paper is the Cartan connection.

%%%%%%%%%%%%%%%%%%%%%%%%%%%%%%%%%%%%%%%%%%%%%%%%%%%%%%%%%% SECTION 2. %%%%%%%%%%%%%%%%%%%%%%%%%%%%%%%%%%%%%%%%%%%%%%%%%%%%%%%%

\Section{Hyper-generalized recurrence}

In this section, we introduce and study a new special Finsler space called hyper generalized recurrent  Finsler  spaces.
 The properties of such spaces are investigated. Some relations between such recurrence and  other  Finsler recurrences are obtained.

\smallskip
For a Finsler manifold  $(M,L)$,  we set the following notations:
\vspace{-8pt}
\begin{eqnarray*}
\stackrel{h}\nabla&:&\text{the $h$-covariant derivatives associated
with Cartan connection }\nabla,\\
\text{Ric} &:& \text{the  horizontal  Ricci curvature tensor of Cartan connection},\\
\text{Ric}_{o} &:& \text{the  horizontal  Ricci vector form defined by }\\
 {\qquad\qquad\qquad}&& g(\text{Ric}_{o} \overline{X},\overline{Y})=\text{Ric}(\overline{X},\overline{Y})   ,\\
 r &:& \text{the horizontal scalar curvature of Cartan connection},\\
G(\overline{X},\overline{Y})\overline{Z} &:=& g(\overline{X},\overline{Z})
\overline{Y}-g(\overline{Y},\overline{Z})\overline{X},\\
C &:=& {R}-\frac{r}{n(n-1)}\, {G}: \text{ the concircular curvature tensor},  \\
{\textbf{G}}(\overline{X},\overline{Y},\overline{Z},\overline{W}) &:=& g(G(\overline{X},\overline{Y})\overline{Z},\overline{W}),\\
{\textbf{C}}(\overline{X},\overline{Y},\overline{Z},\overline{W}) &:=& g(C(\overline{X},\overline{Y})\overline{Z},\overline{W}),\\
{\textbf{R}}(\overline{X},\overline{Y},\overline{Z},\overline{W}) &:=& g(R(\overline{X},\overline{Y})\overline{Z},\overline{W}).
\end{eqnarray*}

A Finsler manifold is said to be horizontally integrable if its
horizonal distribution is completely integrable or, equivalently, if $\widehat{R}=0$.

\smallskip

Firstly, for  a Finsler manifold  of dimension $n\geq3$ with non-zero $h$-curvature tensor
 ${R}$, we define the Kulkarni-Nomizu product $g \wedge \textmd{\text{Ric}}$ of the Finsler metric $g$ and the Ricci curvature tensor $ \emph{\text{Ric}}$
  \cite{R1}:
 \begin{eqnarray}
 % \nonumber to remove numbering (before each equation)
    (g \wedge \textmd{\text{Ric}})(\overline{X},\overline{Y}, \overline{Z}, \overline{W}) &:=&g(\overline{X},\overline{Z}) \textmd{\text{Ric}}(\overline{Y},\overline{W}) +g(\overline{Y},\overline{W})  \textmd{\text{Ric}}(\overline{X},\overline{Z}) \nonumber \\
   && -g(\overline{X},\overline{W}) \textmd{\text{Ric}}(\overline{Y},\overline{Z})
    -g(\overline{Y},\overline{Z}) \textmd{\text{Ric}}(\overline{X},\overline{W}).\label{eq.b1}
 \end{eqnarray}

\begin{rem}\label{rem.1} \textmd{One can show that} $g \wedge g =2 {\bf G}$.
\end{rem}

\begin{defn}\label{def.1a} Let $(M,L)$ be a Finsler manifold  of dimension $n\geq3$ with non-zero $h$-curvature tensor
 ${R}$. Then, $(M,L)$ is called  hyper-generalized recurrent Finsler manifold, denoted by $HGF_{n}$,  if
 $$\stackrel{h}\nabla {\textbf{R}}=A\otimes {\textbf{R}}+ B \otimes (g\wedge \textmd{\text{Ric}}) ,$$
       where $A$ and $B$  are nonzero scalar 1-forms  on $TM$, called the recurrence forms.
\end{defn}

   The following result gives some properties for the hyper-generalized recurrent Finsler manifold.

\begin{thm}\label{thm.3}  Let $(M,L)$ be a horizontally integrable hyper-generalized recurrent Finsler manifold
 of dimension $n\geq3$ with recurrence forms $A$ and $B$. Then, we have\,:\vspace{-0.2cm}
\begin{description}
\item[(a)] if $r\neq0$, then $(M,L)$ is generalized  Ricci recurrent $(\textmd{GRF}_{n})$.
\item[(b)] if $r$ is a nonzero constant, then $(M,L)$ is generalized  2-Ricci recurrent $(\textmd{G(2RF}_{n}))$.
\item[(c)] if $r$ is  non-constant, then the following relation holds
$$(A +(n-2) B)\circ\textmd{\text{Ric}}_{o}=\frac{r}{2}(A+2(n-2)B).$$
\end{description}
\end{thm}

\begin{proof}~ \par

\vspace{5pt}

\noindent $\textbf{(a)}$  Let $(M,L)$ be a horizontally integrable hyper-generalized recurrent Finsler manifold
 of dimension $n\geq3$ with recurrence forms $A$ and $B$. Then, by Definition \ref{def.1a},
 \begin{eqnarray}
 % \nonumber to remove numbering (before each equation)
    (\nabla_{\beta \overline{U}} {\textbf{R}})(\overline{X},\overline{Y}, \overline{Z}, \overline{W}) &=&A(\overline{U}) {\textbf{R}}(\overline{X},\overline{Y}, \overline{Z}, \overline{W})+B(\overline{U})\{g(\overline{X},\overline{Z}) \textmd{\text{Ric}}(\overline{Y},\overline{W})\nonumber \\
   &&  +g(\overline{Y},\overline{W})  \textmd{\text{Ric}}(\overline{X},\overline{Z}) -g(\overline{X},\overline{W}) \textmd{\text{Ric}}(\overline{Y},\overline{Z})\nonumber \\
   &&     -g(\overline{Y},\overline{Z}) \textmd{\text{Ric}}(\overline{X},\overline{W})\}.\label{eq.b2}
 \end{eqnarray}
 Contracting both sides of the above equation with respect to $\overline{Y}$ and $\overline{W}$, noting that $\textmd{\text{Ric}}$ is symmetric (Lemma 2.4 of \cite{Types}), we get
 \begin{eqnarray}
 % \nonumber to remove numbering (before each equation)
    (\nabla_{\beta \overline{U}} \textmd{\text{Ric}})(\overline{X},\overline{Z}) &=&(A(\overline{U})+(n-2)B(\overline{U})) \textmd{\text{Ric}}(\overline{X}, \overline{Z})+r B(\overline{U})g(\overline{X},\overline{Z}) \nonumber\\
    &=&A_{1}(\overline{U}) \textmd{\text{Ric}}(\overline{X}, \overline{Z})+B_{1}(\overline{U})g(\overline{X},\overline{Z}),\label{eq.b3}
\end{eqnarray}
 where $A_{1}:=A+(n-2)B$ and $B_{1}:=rB$. Since $r, A$ and $B$ are nonzero, then
  $A_{1}$ and $B_{1}$ are nonzero. Hence, by Definition 2.2 of \cite{Types}, it follows that $(M,L)$ is  generalized Ricci recurrent.

\bigskip
 \noindent $\textbf{(b)}$ Follows from Equation (\ref{eq.b3}), taking into account that $r$ is nonzero constant.

 \bigskip
 \noindent $\textbf{(c)}$  Since $(M,L)$ is  horizontally integrable, then, using Lemma 2.4 of \cite{Types}, we have
\footnote{$\mathfrak{S}_{\overline{X},\overline{Y},\overline{Z}}$
denotes the cyclic sum over ${\overline{X},\overline{Y},\overline{Z}}$.}
\begin{eqnarray}
% \nonumber to remove numbering (before each equation)
  \mathfrak{S}_{\overline{X},\overline{Y},\overline{Z}}\,\{({\nabla}_{\beta\overline{X}}R)(\overline{Y},\overline{Z}, \overline{W})\}&=&0\label{eq.b4}.
\end{eqnarray}
  Contracting both sides of Equation  (\ref{eq.b3}) with respect to $\overline{X}$ and $\overline{Z}$,  we get
  \begin{eqnarray}
 % \nonumber to remove numbering (before each equation)
    (\stackrel{h}{\nabla} r)(\overline{U}) &=&r(A(\overline{U})+(n-2)B(\overline{U}))+rn B(\overline{U}).\label{eq.b5}
\end{eqnarray}
 Again, contracting both sides of Equation (\ref{eq.b3}) with respect to $\overline{U}$ and $\overline{X}$, taking into account (\ref{eq.b4}) and the symmetry of $\textmd{\text{Ric}}$, we obtain
 \begin{eqnarray}
 % \nonumber to remove numbering (before each equation)
    \frac{1}{2}(\stackrel{h}{\nabla} r)(\overline{Z}) &=&r(A(\textmd{\text{Ric}}_{o}\overline{Z})+(n-2)B(\textmd{\text{Ric}}_{o}\overline{Z}))+r B(\overline{Z}).\label{eq.b6}
\end{eqnarray}
 Now, from  (\ref{eq.b5}) and (\ref{eq.b6}), we obtain
 $$(A +(n-2) B)\circ\textmd{\text{Ric}}_{o}=\frac{r}{2}(A+2(n-2)B).$$
 This completes the proof.
 \end{proof}

\begin{thm}\label{thm.4}  Let $(M,L)$ be a horizontally integrable hyper-generalized recurrent Finsler manifold  of dimension $n\geq3$ with
 recurrence forms $A$ and $B$. If $r$ is a nonzero constant, then we have\textmd{:}
\begin{description}
\item[(a)] the associated 1-forms $A$ and $B$ are related by $A+2(n-1)B=0$,

\item[(b)] $\displaystyle{\frac{r}{n}}$ is an eigenvalue of $\textmd{\text{Ric}}_{o}$ and $\overline{\sigma}$,  $\overline{\rho}$ are eigenvectors
 corresponding to $\displaystyle{\frac{r}{n}}$, where $\overline{\sigma}$ and $\overline{\rho}$ are defined respectively by\, $g(\overline{\sigma},\overline{X}):=A(\overline{X})$ and $g(\overline{\rho},\overline{X}):=B(\overline{X})$.
\end{description}
\end{thm}

\begin{proof}~ \par

\vspace{5pt}

\noindent $\textbf{(a)}$  Follows from Equation (\ref{eq.b5}), using the   assumption that $r$ is a nonzero constant.

\vspace{5pt}

\noindent $\textbf{(b)}$ From Theorem \ref{thm.3}(\textbf{c}) and the fact that $A+2(n-1)B=0$, we conclude that

\begin{equation}\label{eq.b7}
  A({\textmd{Ric}}_{o}\,\overline{X})=\frac{r}{n}A(\overline{X}),   \quad\, \, B(\textmd{\text{Ric}}_{o}\,\overline{X})=\frac{r}{n}B(\overline{X}) .
\end{equation}
 From which, using the symmetry of $\textmd{\text{Ric}}$ and the nondegeneracy of  $g$, we get
 $$\textmd{\text{Ric}}_{o}\,\overline{\sigma}=\frac{r}{n}\,\overline{\sigma},\quad  \, \, \textmd{\text{Ric}}_{o}\,\overline{\rho}=\frac{r}{n}\,\overline{\rho}  .$$
This proves the result.
 \end{proof}

\begin{thm} \label{thm.9} Let $(M,L)$ be a horizontally integrable  hyper-generalized recurrent Finsler manifold  of dimension $n\geq3$ with
 recurrence forms $A$ and $B$. If $r$ is a non-vanishing constant, then we have
  \begin{description}
  \item[(a)] $\mathfrak{S}_{\overline{X},\overline{Y},\overline{Z}} \,\{{(A \otimes {\textbf{R}}+B\otimes (g\wedge\textmd{\text{Ric}}))}(\overline{X},\overline{Y},\overline{Z},\overline{U},\overline{V})\}=0$
  \item[(b)]  $\stackrel{h}{\nabla} A$ and $\stackrel{h}{\nabla} B$  are  symmetric,
  \item[(c)]  $R(\overline{X},\overline{Y}){ \textbf{R}}=0$.
        \end{description}
  \end{thm}

\begin{proof}~ \par

\vspace{5pt}
\noindent $\textbf{(a)}$ Follows from  Definition \ref{def.1a} and  Equation (\ref{eq.b4}).

\vspace{5pt}

\noindent $\textbf{(b)}$ By Definition \ref{def.1a}, we have
\begin{equation}\label{t1}
  \stackrel{h}\nabla\textmd{\text{Ric}}=(A+(n-2)B)\otimes \textmd{\text{Ric}}+ r B \otimes g.
\end{equation}
Again, from the same definition, taking (\ref{t1}) into account, one can show that

\begin{eqnarray*}
% \nonumber to remove numbering (before each equation)
  \stackrel{h}\nabla \stackrel{h}\nabla {\textbf{R}}&=&(\stackrel{h}\nabla A+A\otimes A)\otimes {\textbf{R}}+2rB\otimes B \otimes \textbf{G} \\
  &&+( A\otimes B+ \stackrel{h}\nabla B+B\otimes A+(n-2)B\otimes B)\otimes (g\wedge \textmd{\text{Ric}}).
\end{eqnarray*}
From which, using Lemma 2.4(\textbf{g}) of \cite{Types}, we obtain
\begin{eqnarray}
% \nonumber to remove numbering (before each equation)
   {R}(\overline{U},\overline{V}){\textbf{R}}&=&-({\overline{d}}A)(\overline{U},\overline{V}) {\textbf{R}}-
   ({\overline{d}}B)(\overline{U},\overline{V})(g\wedge \textmd{\text{Ric}}),\label{eq.16a}
\end{eqnarray}
where $({\overline{d}}A)(\overline{U},\overline{V}):=(\stackrel{h}\nabla {A})(\overline{U},\overline{V})-
(\stackrel{h}\nabla {A})(\overline{V},\overline{U})$.

\par

On the other hand, from Theorem \ref{thm.3}\textbf{(a)}, we conclude that
\begin{equation}\label{tt1}
  {\overline{d}}B=-\frac{{\overline{d}}A}{2(n-1)}.
\end{equation}
Hence, (\ref{eq.16a}) and  (\ref{tt1}) yield
\begin{eqnarray}
% \nonumber to remove numbering (before each equation)
   {R}(\overline{U},\overline{V}) {\textbf{R}}&=&-({\overline{d}}A)(\overline{U},\overline{V})\mathcal{H}, \label{eq.17a}
\end{eqnarray}
where  $\,\mathcal{H}:=\{ {\textbf{R}}-\frac{1}{2(n-1)}(g\wedge \textmd{\text{Ric}})\}$. \\
 From Equation (\ref{eq.17a}) and  Lemma 2.4(\textbf{g}) of \cite{Types}, we get \footnote{$\mathfrak{S}_{\overline{U},\overline{V};\,\,\overline{W},\overline{X};\,\, \overline{Y},\overline{Z}}$
denotes the cyclic sum over the three pairs of arguments
$\overline{U},\overline{V};\,\, \overline{W},\overline{X};\,\,\overline{Y},\overline{Z}$.}
 \begin{eqnarray*}
% \nonumber to remove numbering (before each equation)
  \mathfrak{S}_{\overline{U},\overline{V};\,\,\overline{W},\overline{X};\,\,\overline{Y},\overline{Z}}
  \set{({\overline{d}}A)(\overline{U},\overline{V})\mathcal{H}(\overline{W},\overline{X},\overline{Y},\overline{Z})}&=&0 .
\end{eqnarray*}
This, and the fact that $\mathcal{H}(\overline{X},\overline{Y},\overline{Z},\overline{W})=\mathcal{H}(\overline{Z},\overline{W},\overline{X},\overline{Y})$, imply that
\begin{equation}\label{da}
   {\overline{d}}A=0.
\end{equation}
Hence, by Equation (\ref{tt1}), ${\overline{d}}B=0$.

\vspace{5pt}
\noindent $\textbf{(c)}$ Follows from   (\ref{eq.17a}) and (\ref{da}).
\end{proof}

\begin{thm}\label{thm.5}  Let $(M,L)$ be a horizontally integrable hyper-generalized recurrent Finsler manifold
 of dimension $n\geq3$ with recurrence forms $A$ and $B$. If $r=0$, then the following relations hold:
\begin{description}
\item[(a)]$A\circ\textmd{\text{Ric}}_{o} =0, \, B\circ\textmd{\text{Ric}}_{o}=0$,

\item[(b)] $A(R(\overline{X},\overline{Y})\overline{\rho})=0$,
\item[(c)] $\mathfrak{S}_{\overline{X},\overline{Y},\overline{Z}}\{A(\overline{X})B(R(\overline{Y},\overline{Z})\overline{W})\}=0$.
\end{description}
\end{thm}

\begin{proof}~ \par
\vspace{5pt}

\noindent $\textbf{(a)}$ Follows from Equation (\ref{eq.b7}) since $r=0$.

\vspace{5pt}

\noindent $\textbf{(b)}$  Using Equation (\ref{eq.b4}), we have
  \begin{eqnarray}
 % \nonumber to remove numbering (before each equation)
   0 &=&\mathfrak{S}_{\overline{X},\overline{Y},\overline{Z}}\,\{A(\overline{X}) {\textbf{R}}(\overline{Y},\overline{Z}, \overline{U}, \overline{V})\}+\mathfrak{S}_{\overline{X},\overline{Y},\overline{Z}}\,\{B(\overline{X})\{g(\overline{Y},\overline{U}) \textmd{\text{Ric}}(\overline{Z},\overline{V})\nonumber \\
   &&  +g(\overline{Z},\overline{V})  \textmd{\text{Ric}}(\overline{Y},\overline{U})
   -g(\overline{Y},\overline{V}) \textmd{\text{Ric}}(\overline{Z},\overline{U}) -g(\overline{Z},\overline{U}) \textmd{\text{Ric}}(\overline{Y},\overline{V})\}\}.\label{eq.b8a}
 \end{eqnarray}
Contracting both sides of  (\ref{eq.b8a}) with respect to $\overline{Z}$ and $\overline{V}$ and using \textbf{(a)} above,  we obtain
\begin{eqnarray*}
 % \nonumber to remove numbering (before each equation)
   &&\{A(\overline{X})+(n-2)B(\overline{X})\} \textmd{\text{Ric}}(\overline{Y},\overline{U})+rB(\overline{X})g(\overline{Y},\overline{U})\nonumber \\
   &&-\{A(\overline{Y})+(n-2)B(\overline{Y})\} \textmd{\text{Ric}}(\overline{X},\overline{U})+rB(\overline{Y})g(\overline{X},\overline{U})\nonumber \\ && + A(R(\overline{X},\overline{Y})\overline{U})+B(\overline{Y})\textmd{\text{Ric}}(\overline{X},\overline{U})-
   B(\overline{X})\textmd{\text{Ric}}(\overline{Y},\overline{U})=0.\label{eq.b8}
 \end{eqnarray*}
From which, by setting $\overline{U}=\overline{\rho}$ and noting that $B(\overline{X}):=g(\overline{X},\overline{\rho})$ and
$B\circ\textmd{\text{Ric}}_{o}=0$, we conclude that
$$A(R(\overline{X},\overline{Y})\overline{\rho})=0.$$

\vspace{5pt}

\noindent $\textbf{(c)}$ Follows from Equation (\ref{eq.b8a}) by setting $\overline{U}=\overline{\rho}$ and taking into account the fact that
$B\circ\textmd{\text{Ric}}_{o}=0$.
\end{proof}

\section{Conharmonic recurrence}

In this section, we investigate  two types of  Finsler recurrence, namely the conharmonic and generalized conharmonic recurrences. Some relations between such recurrences and other Finsler recurrences are obtained.

\smallskip

\begin{defn}\label{def.r1}
Let $(M,L)$ be Finsler manifold  of dimension $n\geq3$ with nonzero $h$-curvature tensor
 ${R}$. The
$\pi$-tensor field $ {\mathbb{C}}$
 defined by
 \begin{eqnarray}
 \mathbb{C}:=\textbf{R}-\frac{1}{(n-2)}(g\wedge \textmd{{Ric}}) \label{eq.bt2}
 \end{eqnarray}
will be called the conharmonic curvature tensor, ${g\wedge \textmd{{Ric}}}$ being the Kulkarni-Nomizu product of  $g$ and  $\emph{\text{Ric}}$ defined by (\ref{eq.b1}).
\par If the conharmonic curvature tensor ${\mathbb{C}}$ vanishes, then $(M,L)$ is
said to be conharmonically flat.
\end{defn}

It should be noted that the conharmonic curvature tensor in \emph{Riemannian geometry} has been thoroughly investigated by many authors, see for example
 \cite{R3, R1}.
 The above definition is the Finsler version  of such tensor.

\begin{defn}\label{def.2a} Let $(M,L)$ be a Finsler manifold  of dimension $n\geq3$ with nonzero $h$-curvature tensor
 ${R}$. Then, $(M,L)$ is said to be:
 \begin{description}
   \item[(a)] conharmonically recurrent Finsler manifold $({\overline{C}}F_{n})$ if $$\stackrel{h}{\nabla} {\mathbb{C}} = A\otimes {\mathbb{C}},$$
      \item[(b)] generalized conharmonically recurrent Finsler manifold $(G{\overline{C}}F_{n})$ if \,\,$$\stackrel{h}{\nabla} {\mathbb{C}}= A\otimes {\mathbb{C}}+B\otimes G,$$
   \end{description}
   where $A$ and $B$ are nonzero scalar 1-forms  on $TM$,
 positively homogenous of degree zero in the directional argument, called the recurrence forms.

In particular, if \,$\stackrel{h}{\nabla} {\mathbb{C}}=0$, then $(M,L)$ is
called conharmonically symmetric.
\end{defn}

\begin{thm}\label{thm.10}  Let $(M,L)$ be a horizontally integrable hyper-generalized recurrent Finsler manifold
 of dimension $n\geq3$ with recurrence forms $A$ and $B$. Then, we have\,:\vspace{-0.2cm}
\begin{description}
\item[(a)] if $r\neq0$, then $(M,L)$ is generalized  conharmonically  recurrent $({G{\overline{C}}F}_{n})$.
\item[(b)] if $r=0$, then $(M,L)$ is  conharmonically  recurrent $({{\overline{C}}F}_{n})$.
\end{description}
\end{thm}

\begin{proof} The proof follows from Definition \ref{def.r1}, taking into  account Equations (\ref{eq.b2}) and (\ref{eq.b3}).
\end{proof}

 \begin{thm}\label{thm.11}  Let $(M,L)$ be a Finsler manifold  of dimension $n\geq3$ with nonzero $h$-curvature tensor
 ${R}$. If $(M,L)$ satisfies $\emph{\text{Ric}}=\frac{r(n-2)}{2n(n-1)} \, g$, then, we have:\vspace{-0.2cm}
\begin{description}
\item[(a)] $(M,L)$ is concircularly recurrent if and only if it is conharmonically recurrent.

\item[(b)] $(M,L)$ is generalized concircularly recurrent if and only if it is generalized conharmonically recurrent.

 \end{description}
\end{thm}
\begin{proof} The proof follows from the fact that the concircular curvature tensor $C$ and
 the conharmonic curvature tensor $\mathbb{C}$ coincide, under the given assumption.
 \end{proof}

\begin{cor} In a Finsler manifold with nonzero $h$-curvature tensor
 ${R}$ satisfying  $\emph{\text{Ric}}=\frac{r(n-2)}{2n(n-1)} \, g$,  the two  notions of being  concircularly symmetric
    and conharmonically symmetric coincide.
\end{cor}

We end our paper with the following result.

\begin{thm}\label{thm.12}  Let $(M,L)$ be a generalized conharmonically recurrent Finsler manifold
 of dimension $n\geq3$ with recurrence forms $A$ and $B$. Then, we have:\vspace{-0.2cm}
\begin{description}
\item[(a)] if $\, \stackrel{h}{\nabla} \emph{\text{Ric}}=-\frac{(n-2)}{2} \, B \otimes \, g$, then  $(M,L)$ is $HGF_{n}$.

\item[(b)]  if $(M,L)$ is Ricci recurrent with  recurrence  form $A$, then  it is generalized recurrent.

\item[(c)] if $\, \stackrel{h}{\nabla} \emph{\text{Ric}}= A \otimes\emph{\text{Ric}}- \frac{(n-2)}{2}\, B \otimes \, g$, then  $(M,L)$ is recurrent.
 \end{description}
\end{thm}
\begin{proof} ~ \par
\vspace{5pt}

\noindent $\textbf{(a)}$ Since  $(M,L)$ is generalized conharmonically recurrent with recurrence forms $A$ and $B$, then, from Definition \ref{def.2a}(\textbf{b}) and Equation (\ref{eq.bt2}), we have

 \begin{eqnarray}
 \stackrel{h}{\nabla}\mathbb{C}&=&\stackrel{h}{\nabla}\textbf{R}-\frac{1}{(n-2)}(g\wedge \stackrel{h}{\nabla}\textmd{{Ric}}) \nonumber \\
                               &=&A \otimes \{\textbf{R}-\frac{1}{(n-2)}(g\wedge \textmd{{Ric}})\}+B \otimes G \label{dd}.
 \end{eqnarray}
As $\stackrel{h}{\nabla} \emph{\text{Ric}}=-\frac{(n-2)}{2} \, B \otimes \, g$, the above relation reduces to
\begin{eqnarray*}
\stackrel{h}{\nabla}\textbf{R}&=&A \otimes \textbf{R}+D \otimes (g\wedge \textmd{{Ric}}),
 \end{eqnarray*}
where $D:=-\frac{1}{(n-2)}B $.  Consequently,  $(M,L)$ is $HGF_{n}$.

\vspace{5pt}

\noindent $\textbf{(b)}$ Since  $(M,L)$ is  Ricci recurrent   with  recurrence form $A$. Then,  we have

$$\stackrel{h}{\nabla} \emph{\text{Ric}}= A \otimes  \emph{\text{Ric}}.$$
From which, together with Equation (\ref{dd}), the result follows.

\vspace{5pt}

\noindent $\textbf{(c)}$  the proof is similar to that of item (\textbf{a}).
 \end{proof}

\bigskip

\noindent\textbf{Acknowledgments}

\vspace{5pt}
 The author express his sincere
thanks to  Professor Nabil L. Youssef for his
valuable comments and suggestions.

\par
\bigskip

%%%%%%%%%%%%%%%%%%%%%%%%%%%%%%%%%%%%%%%%%%%%%%%%%%%% Refrences %%%%%%%%%%%%%%%%%%%%%%%%%%%%%%%%%%%%%%%%%%%%%%%%%

\providecommand{\bysame}{\leavevmode\hbox
to3em{\hrulefill}\thinspace}
\providecommand{\MR}{\relax\ifhmode\unskip\space\fi MR }
% \MRhref is called by the amsart/book/proc definition of \MR.
\providecommand{\MRhref}[2]{%
  \href{http://www.ams.org/mathscinet-getitem?mr=#1}{#2}
} \providecommand{\href}[2]{#2}

\end{document}